\documentclass{article}
\usepackage{times,authblk,graphicx,german}
\usepackage{amsmath, amsfonts, amssymb, amsthm}
\renewcommand{\theta}{\vartheta}
\renewcommand{\phi}{\varphi}
\renewcommand{\rho}{\varrho}

\newtheorem{Def}{Definition}[section]
\newenvironment{definition}{\begin{Def} \rm}{\end{Def}}
\newtheorem{lemma}[Def]{Lemma}

\newtheorem{theorem}[Def]{Theorem}

\newtheorem{remark}[Def]{Remark}

\newcommand{\Naturals}{{\mathbb N}}

\newcommand{\Reals}{{\mathbb R}}

\newcommand{\tnorm}{\odot}
\newcommand{\tconorm}{\oplus}
\newcommand{\Ltnorm}{\mathbin{\odot_\text{\L}}}
\newcommand{\Ltconorm}{\mathbin{\oplus_\text{\L}}}

\newcommand{\true}{\top}
\newcommand{\false}{\bot}
\DeclareMathOperator{\stneg}{\sim}
\newcommand{\impl}{\rightarrow}
\newcommand{\implc}[1]{\stackrel{#1}{\rightarrow}}
\newcommand{\entails}{\models}
\newcommand{\Q}{$\mathsf{Q}$}
\newcommand{\LGI}{$\mathsf{LGI}$}
\newcommand{\LGIM}{$\mathsf{LGIM}$}
\newcommand{\provesLGI}{\vdash_{\mathsf{LGI}}}
\newcommand{\provesLGIM}{\vdash_{\mathsf{LGIM}}}
\newcommand{\Add}[1]{\hspace{0.5em}\text{\rm \footnotesize #1}}
\newcommand{\Rule}[1]{\text{\rm\footnotesize (#1)} \hspace{1em}}

\newcommand{\average}[1]{\varnothing(#1)}
\newcommand{\averagealone}{\varnothing}
\newcommand{\LPihalf}{\L$\Pi\frac 1 2$}

\begin{document}


\title{Reasoning with graded information: \\
the case of diagnostic rating scales in healthcare
\thanks{Preprint of an article published by Elsevier in the {\sl Fuzzy Sets and Systems} {\bf 298} (2016), 207-221. It is available online at: {\tt
https://www.sciencedirect.com/science/article/pii/ S0165011415005370}.}}



\author{Thomas Vetterlein}

\affil{Johannes Kepler University, Linz, Austria; \\
{\tt thomas.vetterlein@jku.at}}

\author{Anna Zamansky}

\affil{University of Haifa, Israel; \\
{\tt annazam@is.haifa.ac.il}}

\maketitle


\begin{abstract}

In medicine one frequently deals with vague information. As a tool for reasoning in this area, fuzzy logic suggests itself. In this paper we explore the applicability of the basic ideas of fuzzy set theory in the context of medical assessment questionnaires, which are commonly used, for instance, to support the diagnosis of psychological disorders.

The items of a questionnaire are answered in a graded form; patients are asked to choose an element on a linear scale. The derived diagnostic hypotheses are graded as well. This leads to the question whether there is a logical formalism that is suitable to capture the score calculation of medical assessment questionnaires and thereby provides a mathematical justification of the way in which the calculation is typically done.

We elaborate two alternative approaches to this problem. First, we follow the lines of mathematical fuzzy logic. For the proposed logic, which can deal with the formation of mean values, we present a Hilbert-style deduction system. In addition, we consider a variant of the prototype approach to vagueness. In this case we are led to a framework for which to obtain a logical calculus turns out to be difficult, yet our gain is a model that is conceptually comparably well-justifiable.


\end{abstract}


\section{Introduction}
\label{sec:introduction}

Fuzzy set theory and fuzzy logic have originally been developed with the intention to overcome the particular difficulties that arise when an application requires the evaluation of \emph{vague} information. Here, we call a property vague if it cannot be sharply delimited from its opposite. By default, a property that is applicable to a certain class of objects or processes is thought of as being dichotomous; it serves us to distinguish those cases in which it holds from those in which it does not hold. However, considering a situation in more detail, we may notice that it might not under all circumstances be possible to decide if a property applies or not. In fact, a characteristic feature of vague properties is the presence of borderline cases.

When wondering about examples of vagueness, there is one field that offers an apparently unlimited amount of notions that belong to this category. In fact, in medicine, it might, conversely, be found difficult to find descriptions of a different kind. A property describing the state of some patient is usually vague; we cannot use it positively or negatively under all possible circumstances. More specifically, we think about signs and symptoms of human diseases. Consider, e.g., the property of ``having high fever''. Below a temperature of, e.g., $38.5^\circ$, nobody would speak about high fever; above, e.g., $39.5^\circ$ it is not questionable to speak about high fever; the remaining cases are borderline.

A common approach to deal with borderline cases is to extend the two-element set of truth values to a continuous-valued one; the two values ``false'' and ``true'' are replaced by the real unit interval. The property of having high fever, e.g., may be conveniently described by a fuzzy set mapping each possible temperature to such a generalised truth value. This is the very idea of fuzzy set theory \cite{zadeh1965fuzzy}, which, as far as the pure use of grades is concerned, is intuitively quite convincing.

We may ask to which extent the idea has been established as a tool for reasoning in medicine. The situation is somewhat ambiguous. Fuzzy logic has been applied, under this name, in the framework of several projects concerned with automated reasoning about medical information. In fact, in this context, its use seems to be clearly implied, given the fact that most of the processed information is vague. Already the system that is often called the ``grandfather'' of medical expert systems, MYCIN, was based on a continuous set of degrees \cite{shortliffe1975MYCIN}. System of this or a similar kind, however, are not often found in routine use in healthcare. To establish the principles of fuzzy logic in medicine -- principles that are not unquestioned even in the community of fuzzy logicians -- is certainly hard. At least in some cases, the basic ideas do have been welcomed. For instance, the programming language Arden Syntax, designed for a platform-independent representation of medical kn
 owledge, has been extended to simplify the treatment of fuzzy sets and is now called Fuzzy Arden Syntax \cite{vetterlein2010arden}. Furthermore, MONI, a decision support system based on a simple fuzzy logic, is at present in use at the General Hospital of Vienna; its purpose is the automated detection of hospital-acquired infections \cite{debruin2012MONI}.

Given the limited presence of fuzzy logic in medical decision support, the aspect that we address in the present paper provides a remarkable contrast. Grades are not just used within computer-based decision support systems; and clinicians do not necessarily consider grades as somewhat academic. In another context gradedness of information comes into play quite naturally.

The \emph{assessment of symptoms} is an essential aspect of the diagnostic procedure for various disorders, in particular in psychiatry and psychology \cite{uexkull2008theorie}. We recall that, as opposed to signs, which are objective phenomena detected by the clinician, symptoms are subjective experiences reported by the patient, such as a complaint of pain or depressed feeling. Thus, with symptoms, clinicians must rely on the patient's self-report, with no objective tests being available to confirm or rule out them \cite{kessler2000methodological}. This in turn opens up the possibility that patients report their symptoms autonomously. The use of questionnaires has contributed to a reduction of the working load in healthcare \cite{heymann2003diagnosebezogene,tritt2007entwicklung}. 

As a key element that we find in this context, the occurring questions can be considered as vague. In accordance with this observation, the answers are given in the form of grades. Consider, for instance, the ICD-10 Symptom Rating (ISR) questionnaire \cite{tritt2007entwicklung}, which was created for the assessment of several psychological disorders. In order to evaluate the depressive syndrome, an answer to the item ``I feel down and depressed'' is required. To this end, the patient is asked to choose an element on a five-element linear scale, ranging from ``0 - does not apply'' to ``4 - applies extremely''. Apparently, we can understand the question as vague and the patient makes a choice to which \emph{degree}, from his point of view, the indicated statement applies to his actual state.

The question on which the present work is based is now: how are the degrees further processed? Commonly, the answers to all questions are aggregated to a single value; questionnaires that are used in routine healthcare are commonly evaluated by the calculation of \emph{rating scores}. Thus the question is which aggregation method is used for which reason. Ideally, questionnaires are designed and evaluated on the basis of well justified principles. In fact, within medical computer science, the topic has found an increasing interest during recent years. A large volume of works focuses on developing theories for a \emph{statistical} analysis of questionnaire data, such as classical test theory or item response theory; see, e.g., \cite{reckase1997past,mcdonald2000basis,devellis2006classical}.

These approaches might, on the one hand, provide useful insights into an optimal score calculation in questionnaires, considering, e.g., the weight and mutual dependency of items in a proper way. In practice, on the other hand, a procedure seems often to be chosen on pragmatic grounds. Consider, for instance, the development of the afore-mentioned ISR questionnaire. In this case, a panel of experts decided on the basis of their medical expertise. They voted on the method of calculating the total score and chose the mean value \cite{tritt2007entwicklung}. The experts also discussed the possibility of adding weights to the items, but decided against it. Other examples of questionnaires that employ similar aggregation methods include the Hamilton Depression Rating Scale and Zung Self-Rated Depression Scale, the RAND SF-36-item, and many more. In our work, we mainly restrict our attention to the mean value, yet having the intention to propose a framework general enough to take account also for other methods of calculating scores. 

The aim of the proposed framework is to examine the problem of score computation from a logical angle. The aim of questionnaire evaluation is deriving from a set of degrees a new degree -- the total score. Many-valued logics provide mathematical tools for dealing with properties that are endowed with truth degrees. Accordingly, our goal is to propose a ``logic of questionnaires'': given a patient's answers to questionnaire items, the formalism should support deriving a total score, possibly taking into account further implicit assumptions.

We provide in this paper two alternative approaches to this problem. The first approach is perhaps the more naive one. We proceed in a way that has become common in mathematical fuzzy logic \cite{hajek1998MFL}. Accordingly, we process the occurring degrees regardless of their origin and define ways of connecting them just as required by our application. However, ``mean-value operations'' do not belong to those operations that are typically included in a fuzzy logic. An exception is the so-called Compensatory Fuzzy Logic discussed in \cite{bouchet2011arithmeticmean}. Moreover, there are several fuzzy logics that are strong enough to allow the definition of mean values. For instance, B.\ Gerla's Rational \L ukasiewicz Logic corresponds to divisible MV-algebras. For each $n$, the logic has a connective interpreted by the division by $n$ and hence for each $n$ formulas, a formula exists that is assigned the mean value of its constituents \cite{gerla2001rationalLL}. On a similar idea, Kukkurainen's and Turunen's formalism presented in \cite{kukkurainen2002similarityreasoning} is based. A further, particular strong fuzzy logic is \LPihalf\ \cite{esteva2001LP}, a combination of \L ukasiewicz and product logic, enriched by the constant $\frac 1 2$. Any rational value in $[0,1]$ is expressible in it and consequently again, mean values can be formed. Here, we propose a logic that is syntactically as scarce as possible, but rich enough for our intended application. We present a sound and complete corresponding Hilbert-style calculus.

Our approach might possess a considerable potential in regards to its formal development. As regards its interpretation and justification, however, it is subjected to the criticism that we find in the ongoing discussion on the foundations of mathematical fuzzy logic and its so-called design choices; see, e.g., \cite{fermueller2011conversation}. For this reason, we propose in addition an alternative approach, aiming at a more appropriate account for the vagueness in assessment questionnaires. The starting point is the well-known idea of modelling the \emph{prototypes} of a vague property by subsets of a metric space. A formalism based on this approach is, e.g., \cite{lawry2009prototype}. We take, however, the set of \emph{counterexamples} of a vague property into account as well; we accordingly deal with a pair of two subsets in a metric space. Also this idea has been exploited, in particular in Nov\' ak's work on the modelling of linguistic expressions \cite{novak2008trichotomous}. The formalism that we propose here is tailored to the emulation of the mean-value score calculation and unrelated to mathematical fuzzy logic. We specify, however, only a semantic framework; how the reasoning in it can be axiomatised remains an open question.

The rest of the paper is organized as follows. Section \ref{sec:fuzzy-logic-with-mean-value} presents a fuzzy logic dealing with mean values. Section \ref{sec:prototypes-counterexamples} is devoted to an alternative approach formalising vague information based on prototypes and counterexamples. The final Section \ref{sec:conclusion} provides a summary and points out remaining challenges.

\section{A fuzzy logic with a mean-value operation}
\label{sec:fuzzy-logic-with-mean-value}

We shall present in this section a particular fuzzy logic. We note that ``fuzzy logic'' is understood in the mathematical sense, that is, we speak about logics that are based on an extended, linearly ordered set of truth values \cite{hajek1998MFL}. Mean values are not commonly dealt with in fuzzy logic; we propose here a particular adaptation for the application under consideration.

When designing our logic, however, we will not just care about including the possibility of calculating mean values. If we did so, we could, for instance, opt for Rational \L ukasiewicz Logic \cite{gerla2001rationalLL} or \LPihalf{} \cite{esteva2001LP}. We rather choose an approach with the aim of alleviating a general problem of fuzzy logic: the interpretation of the implication connective. Namely, implications will be endowed with explicit grades and these grades will not be subjected to logical connectives.

\subsection{Fuzzy logic and the role of the implication}
\label{subsec:role-of-implication}

The statements of our logic will be intended to refer to medical facts and to explain the state of a patient. Consequently, all properties with which we deal are assumed to be graded. Furthermore, we follow a main principle of mathematical fuzzy logic: the degree of compound expressions is calculated from the degree of the components.

As common in fuzzy logic, the connectives are chosen pragmatically. We include the possibly most significant connectives: the minimum, the maximum, and the standard negation. In addition to the minimum, the ``weak conjunction'', we include a ``strong conjunction'', interpreted by a t-norm.

Furthermore, it will be possible to make comparisons with regard to degrees in the following way. For two formulas $\alpha$ and $\beta$, the expression $\alpha \implc{1} \beta$ is two-valued and has the meaning that the truth degree of $\beta$ is at least as large as the degree of $\alpha$. That is, if $\alpha$ is assigned $s$ and $\beta$ is assigned $t$, we require $s \leq t$. This relationship can be weakened; $\alpha \implc{c} \beta$, where $0 \leq c \leq 1$, says that the inquality holds up to a tolerance value of $1-c$: the truth degree of $\beta$ is at least as large as the degree of $\alpha$ reduced by $1-c$. That is, we require $s - (1-c) \leq t$ in this case.

We note that we do not assign continuous truth degrees to expressions of the form $\alpha \implc{c} \beta$; we rather adopt the approach that the comparison of truth degrees leads to a positive or to a negative result. Consequently, we also do not allow to nest implicative relationships.

There are two reasons that motivate our decision. First, the approach is suitable for our application; we get not less and not (much) more than what we need. Second, we avoid difficulties concerning the interpretation of formal statements. Instead of working with expressions of the form $\alpha \implc{c} \beta$, we could alternatively include truth constants as well as the connective $\impl$ to our language, interpreted by the residuum corresponding to the strong conjunction. However, in this case expressions involving the implication could be nested and statements with a doubtful meaning like $(\alpha \impl \beta) \impl \beta$ would be possible. We are tempted to read ``if $\alpha$ is stronger than $\beta$, then $\beta$ holds'', but apart from the fact that this statement is itself inacceptable, it would not at all reflect the meaning of the formula. We could rather say ``whenever a property $\gamma$ is strong enough such that $\alpha$ and $\gamma$ together are stronger than $\beta$, then $\gamma$ is stronger than $\beta$''. We consider such constructs as not only artificial but in fact unnecessary. We should certainly stress that we refer here to a particular context; there are frameworks in which the residual implication does play a clear role \cite{fermueller2008overview}. But such contexts are much different from ours.

Accordingly, we use a two-level approach. The graded implications are on the inner level and they are crisp: they either hold or do not hold. On the outer level, graded implications will be allowed to be combined by the classical logical connectives.

\subsection{A fuzzy logic based on graded implications}

As our first step, we introduce in this subsection a calculus that does not yet deal with mean values; an extended calculus will be presented subsequently.

Our calculus will be called the {\it Logic for Graded Implications}, or \LGI{} for short. Let us specify the syntax for \LGI. We start with a countable set $\phi_0, \ldots$ of variables and the two constants $\false$, standing for clear falsity, and $\true$, standing for full truth. A {\it basic expression} is built up from variables and constants by means of the binary connectives $\land, \lor, \tnorm$ and the unary connective $\stneg$.

For the set of truth degrees, we make the common choice, using the real unit interval $[0,1]$. By a {\it graded implication} of \LGI{} we mean a triple consisting of two basic expressions $\alpha$ and $\beta$ and a real number $c \in [0,1]$, denoted by
\[ \alpha \implc{c} \beta. \]

Finally, a \emph{formula} is built up from graded implications by means of the binary connectives $\land, \lor$ and the unary connective $\lnot$. We call the latter the {\it outer} connectives; their intended meaning is the classical ``and'', ``or'', and ``not'', respectively. As usual, we write $\Phi \impl \Psi$ for $\lnot\Phi \lor \Psi$.

As common in fuzzy logic, we allow the usage of two different conjunctions, the ``weak'' one $\land$ and the ``strong'' one $\tnorm$. Whereas the former will be interpreted by the minimum, the latter will be interpreted by a continuous t-norm. By a {\it t-norm}, we mean a binary operation on the real unit interval that is associative, commutative, having $1$ as a neutral element, and isotone in each argument; see, e.g., \cite{klement2000tnormbook}. An example is the \L ukasiewicz t-norm:
\[ c \Ltnorm d \;=\; (c+d-1) \vee 0, \quad\text{ $c, d \in [0,1]$;} \]
further examples are the product and the G\" odel t-norm. For what follows, we fix a continuous t-norm $\tnorm$.

Furthermore, we denote the t-conorm associated with $\tnorm$ by $\tconorm$, that is,
\[ c \tconorm d \;=\; 1-((1-c) \tnorm (1-d)), \quad\text{ $c, d \in [0,1]$.} \]
The \L ukasiewicz t-conorm is given by $c \Ltconorm d = (c+d) \wedge 1$.

To increase readability, we will use for the interpretation of the connectives the same symbol as for the connectives themselves. We write, for instance, $\tnorm$ in both cases; furthermore, we denote the minimum and maximum of numbers $c, d \in [0,1]$ by $c \wedge d$ and $c \vee d$, respectively; and we put $\stneg c = 1 - c$.

An {\it evaluation} is a mapping $v$ from the set of basic expressions to the real unit interval $[0,1]$ preserving the connectives, $\land, \lor, \tnorm, \stneg$, the latter being interpreted, respectively, by the minimum, maximum, the t-norm $\tnorm$, and the standard negation. An evaluation $v$ is said to {\it satisfy} a graded implication $\alpha \implc{c} \beta$ if
\[ v(\alpha) \leq v(\beta) + \stneg c; \]
and we extend the definition of satisfaction to all formulas by interpreting the outer connectives according to classical propositional logic.

Note that we may equivalently say that an evaluation $v$ satisfies a graded implication $\alpha \implc{c} \beta$ if $v(\alpha) \Ltnorm c \leq v(\beta)$. From this point of view, the \L ukasiewicz t-norm is assigned a special role. In fact, it is exactly this t-norm that allows the interpretation of $c$ as a tolerance value, in the sense explained above (see Subsection \ref{subsec:role-of-implication}). We will not discuss here the question whether we could replace $\Ltnorm$ by a different t-norm.

A {\it theory} $\mathcal T$ is a set of formulas. We say that $\mathcal T$ {\it semantically entails} a formula $\Phi$ if, whenever an evaluation $v$ satisfies all elements of $\mathcal T$, $v$ also satisfies $\Phi$. We write ${\mathcal T} \entails \Phi$ in this case.

We now by proceed defining a proof system.

\begin{definition}
The calculus \LGI{} consists of the following axioms and rules:

\vspace{1.9ex}

\noindent any formula arising from a tautology of classical propositional logic by a uniform replacement of the variables by graded implications;

\vspace{1.9ex}

\noindent for any formulas $\Phi$ and $\Psi$ the rule
\[ \Rule{MP} \frac{\Phi \quad \Phi \impl \Psi}{\Psi}; \]
for any basic expressions $\alpha$, $\beta$, $\gamma$ and $c, d \in [0,1]$ the axioms

\vspace{1.9ex}

\begin{minipage}{0.48\textwidth}
$\Rule{$\land_1$} (\alpha \implc{d} \beta) \land (\alpha \implc{d} \gamma)
\impl (\alpha \implc{d} \beta \land \gamma)$

$\Rule{$\land_2$} \alpha \land \beta \implc{1} \alpha$

$\Rule{$\land_3$} \alpha \land \beta \implc{1} \beta$
\end{minipage}
\begin{minipage}{0.48\textwidth}
$\Rule{$\lor_1$} (\alpha \implc{d} \gamma) \land (\beta \implc{d} \gamma)
\impl (\alpha \lor \beta \implc{d} \gamma)$

$\Rule{$\lor_2$} \alpha \implc{1} \alpha \lor \beta$

$\Rule{$\lor_3$} \beta \implc{1} \alpha \lor \beta$
\end{minipage}

\vspace{1.9ex}

\begin{minipage}{0.5\textwidth}
$\Rule{$\tnorm_1$} (\true \implc{c} \alpha) \land (\true \implc{d} \beta)
\impl (\true \implc{c \tnorm d} \alpha \tnorm \beta)$

$\Rule{$\tnorm_2$} (\alpha \implc{c} \false) \land (\beta \implc{d} \false)
\impl (\alpha \tnorm \beta \implc{c \tconorm d} \false)$

$\Rule{$\tnorm_3$} \true \implc{1} \true \tnorm \true$
\end{minipage}
\begin{minipage}{0.46\textwidth}
$\Rule{$\stneg_1$} (\alpha \implc{d} \beta)
\impl (\stneg \beta \implc{d} \stneg\alpha)$

$\Rule{$\stneg_2$} \stneg\stneg\alpha \implc{1} \alpha$

$\Rule{$\stneg_3$} \alpha \implc{1} \stneg\stneg\alpha$
\end{minipage}

\vspace{1.9ex}

\begin{minipage}{0.48\textwidth}
$\Rule{$\true$} \alpha \implc{1} \true$

$\Rule{$\false$} \false \implc{1} \alpha$

$\Rule{0} \alpha \implc{0} \beta$

$\Rule{$c$} \alpha \implc{c} \alpha$

$\Rule{inkons} \neg(\true \implc{c} \false),
\Add{where $c > 0$}$
\end{minipage}
\begin{minipage}{0.48\textwidth}
$\Rule{trans$_1$} (\alpha \implc{c} \beta) \land (\beta \implc{d} \gamma) \impl (\alpha \implc{c \Ltnorm d} \gamma)$

$\Rule{trans$_2$} (\alpha \implc{c} \false) \land (\true \implc{d} \beta) \impl
(\alpha \implc{c \Ltconorm d} \beta);$

$\Rule{lin$_1$} (\alpha \implc{1} \beta) \lor (\beta \implc{1} \alpha)$

$\Rule{lin$_2$} (\true \implc{d} \alpha) \lor (\alpha \implc{\stneg d} \false)$.

\end{minipage}

\vspace{1.9ex}

We define the notion of a {\it proof} in \LGI{} of a formula $\Phi$ from a theory $\mathcal T$ as usual; we write ${\mathcal T} \provesLGI \Phi$ if there is one.
\end{definition}

Let us have an informal look at the rules. The rules $(\land_1)$--$(\land_3)$ characterise the weak conjunction $\land$ in the usual way. Similar rules hold for the strong conjunction $\tnorm$ only in particular cases, see $(\tnorm_1)$--$(\tnorm_3)$. Moreover, the rules $(\lor_1)$--$(\lor_3)$ characterise the disjunction,  $(\stneg_1)$--$(\stneg_3)$ characterise the involutive negation.

$(\true)$ and $(\false)$ define the truth constants as the bottom and the top element, respectively. As regards the rule $(0)$, recall that the value $c$ attached to a graded implication is such that $1-c$ is the allowed tolerance. If $c = 0$, the tolerance is $1$ and the implication holds independently of the involved truth degrees. The rules $($c$)$ and (trans$_1$) express the reflexivity and transitivy of the implication relation. The rule (trans$_2$) refers to transitivity as well; it implies that if $\alpha$ is assigned $r$, $\beta$ is assigned $s$, and $d \in [0,1]$ is such that $r \leq d$ and $d \leq s$, we can conclude $r \leq s$.

Note next that $\true \implc{c} \false$ cannot hold unless $c = 0$, as established by (inkons). Finally, there are two rules dealing with the linearity of the truth degrees. By (lin$_1$), given the truth degrees $r$ and $s$ of two formulas, either $r \leq s$ or $s \leq r$. (lin$_2$) expresses the fact, given the truth degree $r$ of any formula $\alpha$ and any $d \in [0,1]$, either $d \leq r$ or $r \leq d$.

\begin{theorem} \label{thm:soundness-LGI}
Let $\mathcal T$ be a theory and $\Phi$ a formula of \LGI. If ${\mathcal T} \provesLGI \Phi$, then ${\mathcal T} \entails \Phi$.
\end{theorem}

\begin{proof}
It is not difficult to check that all rules are sound.
\end{proof}

We next note that lowering the degree of a graded implication leads, as intended, to a weaker statement.

\begin{lemma} \label{lem:making-weight-smaller}
For any basic expressions $\alpha$ and $\beta$, we can prove in \LGI{} $(\alpha \implc{d} \beta) \impl (\alpha \implc{c} \beta)$ if $c \leq d$.
\end{lemma}

\begin{proof}
By rule ($c$), we have $\beta \implc{1+c-d} \beta$; so the assertion follows from (trans$_1$).
\end{proof}

Note that we can assign in our logic truth degrees explicitly. In fact, $\alpha \implc{\stneg c} \false$ is satisfied by an evaluation $v$ if and only if $v(\alpha) \leq c$; similarly, $\true \implc{c} \alpha$ is satisfied if and only if $c \leq v(\alpha)$.

For a basic expression $\alpha$ and $c \in [0,1]$, we shall denote by $\tau(\alpha,c)$ the following set of formulas: 
$$\tau(\alpha,c)=\{\true \implc{t} \alpha\ | where\ t \in [0,1]\ s.t.\  t < c\}\cup\{\alpha \implc{\stneg t} \false\ |\ where\ t \in [0,1] \ s.t.\  t > c\}$$

We say that $\tau(\alpha,c)$ is provable in \LGI{} from a theory $\mathcal T$ if this is the case for any element of $\tau(\alpha,c)$.

\begin{lemma} \label{lem:explicit-values-preserved-under-connectives}
Let $\mathcal T$ be a theory, let $\alpha, \beta$ be basic expressions, and let $c, d \in [0,1]$. If $\tau(\alpha,c)$ and $\tau(\beta,d)$ are provable from $\mathcal T$ in \LGI{}, then so are $\tau(\alpha \land \beta, c \wedge d)$, $\tau(\alpha \lor \beta, c \vee d)$, $\tau(\alpha \tnorm \beta, c \tnorm d)$, and $\tau(\stneg\alpha, \stneg c)$.
\end{lemma}

\begin{proof}
Assume that $\tau(\alpha,c)$ and $\tau(\beta,d)$ are provable from $\mathcal T$ in \LGI{}. We presume in the sequel that $c$ and $d$ are distinct from $0$ and $1$; if $c$ or $d$ is $0$ or $1$, some of the arguments are to be omitted.

To see that we can derive $\tau(\alpha \land \beta, c \wedge d)$ from $\mathcal T$ in \LGI{}, let $t < c \wedge d$. Then $t < c$ and $t < d$ and hence $\true \implc{t} \alpha$ and $\true \implc{t} \beta$ and by ($\wedge_1$) $\true \implc{t} \alpha \land \beta$. Let now $t > c \wedge d$. If then $t > c$, we conclude from $\alpha \implc{\stneg t} \false$ by ($\wedge_2$) that $\alpha \land \beta \implc{\stneg t} \false$. If $t > d$, we draw the same conclusion from $\beta \implc{\stneg t} \false$ by ($\wedge_3$).

Similarly, we derive $\tau(\alpha \lor \beta, c \vee d)$.

To see that $\tau(\alpha \tnorm \beta, c \tnorm d)$ is derivable as well, let $t < c \tnorm d$. By the continuity of $\tnorm$, there an $r < c$ and $s < d$ such that $t = r \tnorm s$. From $\true \implc{r} \alpha$ and $\true \implc{s} \beta$, we then conclude $\true \implc{t} \alpha \tnorm \beta$ by ($\tnorm_1$). Similarly, we derive $\alpha \tnorm \beta \implc{\stneg t} \false$ for $t > c \tnorm d$.

Finally, to see that $\tau(\stneg\alpha, \stneg c)$ is derivable, let $t < \stneg c$. Then $\stneg t > c$, hence $\alpha \implc{t} \false$ and, by ($\stneg_1$), $\true \implc{t} \stneg\alpha$. Similarly, for $t < \stneg c$, we derive $\stneg\alpha \implc{\stneg t} \false$.
\end{proof}

\begin{lemma} \label{lem:explicit-values-fulfil-graded-implication}
Let $\mathcal T$ be a theory, let $\alpha, \beta$ be basic expressions, and let $c, d \in [0,1]$. Assume that $\tau(\alpha,c)$ and $\tau(\beta,d)$ are provable from $\mathcal T$ in \LGI{}. Then ${\mathcal T} \provesLGI \alpha \implc{r} \beta$ if $r < 1 - c + d$, and ${\mathcal T} \provesLGI \neg(\alpha \implc{r} \beta)$ if $r > 1 - c + d$.
\end{lemma}

\begin{proof}
Let $r < 1 - c + d$. Assume first that $c < 1$ and $d > 0$. Let $s > c$ and $t < d$ be such that $r = 1 - s + t$. From $\alpha \implc{\stneg s} \false$ and $\true \implc{t} \beta$, we infer $\alpha \implc{r} \beta$ by (trans$_2$) because $\stneg s \Ltconorm t = r$. Assume second that $c = 1$. Then $d > 0$; we put $s = 1$ and $t = r < d$ and by (0) we may argue as before. Assume third that $d = 0$. Then $c < 1$; we put $s = 1-r > c$ and $t = 0$ and again by (0) we argue as in the first case.

Let $r > 1 - c + d$. Note that then $c > 0$ and $d < 1$. Thus there are $s < c$ and $t > d$ such that $r > 1 - s + t$. From $\true \implc{s} \alpha$, $\,\alpha \implc{r} \beta$, and $\beta \implc{\stneg t} \false$, we infer $\true \implc{s \Ltnorm r \Ltnorm \stneg t} \false$ by (trans$_1$). As $s \Ltnorm r \Ltnorm \stneg t > 0$, the second part follows by (inkons).
\end{proof}

We now turn to the question of completeness of our calculus \LGI{} with regard to its intended semantics. The following statement might be seen as analogous to the type of completeness that was first proposed by Pavelka in the context of fuzzy logic with evaluated syntax \cite{pavelka1979onfuzzylogic}; see also \cite{novak1999principlesfuzzylogic,hajek1998MFL}.

We need one further additional result on \LGI. The fact expressed in the following lemma is well-known but would, by default, require transfinite induction in the present context. We provide a proof in order to demonstrate that we can do without.

\begin{lemma} \label{lem:theorem-extension}
Let $\mathcal T$ be a theory consisting of graded implications, and let $\Phi$ be a formula such that $\Phi$ is not provable from $\mathcal T$. Then there is a complete theory $\bar{\mathcal T}$ such that $\mathcal T \subseteq \bar{\mathcal T}$ and $\Phi$ is not provable from $\bar{\mathcal T}$ either.
\end{lemma}

\begin{proof}
Let $(\alpha_i, \beta_i, c_i)$, $i \in \Naturals$, be an enumeration of all triples consisting of two basic expressions and a rational element of $[0,1]$. Let ${\mathcal T}_0 = \mathcal T$; for each $i = 0, \ldots$, choose ${\mathcal T}_{i+1}$ to be either ${\mathcal T}_i \cup \{ \alpha_i \implc{c_i} \beta_i \}$ or ${\mathcal T}_i \cup \{ \lnot{(\alpha_i \implc{c_i} \beta_i)} \}$, to the effect that $\Phi$ is not provable from ${\mathcal T}_{i+1}$; and let ${\mathcal T}' = \bigcup_i {\mathcal T}_i$.

By Lemma \ref{lem:making-weight-smaller}, \LGI{} proves $(\alpha \implc{d} \beta) \impl (\alpha \implc{c} \beta)$ for any basic expressions $\alpha$ and $\beta$ and any $c, d \in [0,1]$ such that $c \leq d$. Accordingly, we add to ${\mathcal T}'$ any graded implication $\alpha \implc{c} \beta$ such that $\alpha \implc{d} \beta$ is in ${\mathcal T}'$ for some $d > c$. Likewise, we add $\lnot{(\alpha \implc{c} \beta)}$ whenever $\lnot{(\alpha \implc{d} \beta)}$ is in ${\mathcal T}'$ for some $d < c$. Let ${\mathcal T}''$ be the resulting theory. Obviously, $\Phi$ is not provable from ${\mathcal T}''$.

For any pair of basic expressions $\alpha, \beta$, we then have that for at most one $c \in [0,1]$ neither $\alpha \implc{c} \beta$ nor $\lnot{(\alpha \implc{c} \beta)}$ is in ${\mathcal T}''$. Hence may again successively extend ${\mathcal T}''$ to a theory ${\mathcal T}'''$ such that $\Phi$ is not provable from ${\mathcal T}'''$ and, for any $\alpha, \beta$ and any $c \in [0,1]$, either $\alpha \implc{c} \beta$ or $\lnot{(\alpha \implc{c} \beta)}$ is in ${\mathcal T}'''$. Finally, closing ${\mathcal T}'''$ under the formulas provable from it, we get a theory as desired.
\end{proof}

\begin{theorem} \label{thm:completeness-LGI}
Let $\mathcal T$ be a theory consisting of graded implications, and let $\zeta \implc{e} \eta$ be a graded implication of \LGI. If ${\mathcal T} \entails \zeta \implc{e} \eta$, then ${\mathcal T} \provesLGI \zeta \implc{t} \eta$ for any $t < e$.
\end{theorem}

\begin{proof}
Assume to the contrary that, in \LGI, there is a $t < e$ such that there is no proof of $\zeta \implc{t} \eta$ from $\mathcal T$. We shall show that there is an evaluation satisfying all elements of $\mathcal T$ but not $\zeta \implc{e} \eta$.

By Lemma \ref{lem:theorem-extension}, we can extend $\mathcal T$ to a theory $\bar{\mathcal T}$ such that also $\bar{\mathcal T}$ does not prove $\zeta \implc{t} \eta$ in \LGI{} and, for each formula $\Phi$, either $\Phi$ or $\lnot\Phi$ is in $\bar{\mathcal T}$.

Then, for each basic expression $\alpha$, there is by (lin$_2$) and Lemma \ref{lem:making-weight-smaller} a $v(\alpha) \in [0,1]$ such that $\tau(\alpha, v(\alpha))$ is provable from $\bar{\mathcal T}$. By Lemma \ref{lem:explicit-values-preserved-under-connectives}, $v$ is an evaluation. Moreover, by Lemma \ref{lem:explicit-values-fulfil-graded-implication}, $\bar{\mathcal T}$ proves a graded implication $\alpha \implc{c} \beta$ if $c < 1 - v(\alpha) + v(\beta)$, and does not prove it if $c > 1 - v(\alpha) + v(\beta)$. Consequently, if $\alpha \implc{c} \beta$ is contained in $\bar{\mathcal T}$, we have $c \leq 1 - v(\alpha) + v(\beta)$, that is, $v(\alpha) \leq v(\beta) + \stneg c$, and $\alpha \implc{c} \beta$ is satisfied by $v$. Moreover, since $\bar{\mathcal T}$ does not prove $\zeta \implc{t} \eta$, we have that $1 - v(\zeta) + v(\eta) \leq t < e$. This in turn means that $v$ does not satisfy $\zeta \implc{e} \eta$.
\end{proof}

Having restricted to a Pavelka-style completeness, we note that, in principle, the possibility exists to extend the calculus such that strong completeness, in the usual sense, can be established. Namely, we may add the following infinitary rule to \LGI:
\[ \frac{(\true \implc{c} \alpha) \impl \Phi \Add{for all $c > d$}
\quad\quad
(\alpha \implc{\stneg d} \false) \impl \Phi}{\Phi}, \]
where $\Phi$ is a formula, $\alpha$ is a basic expression, and $d \in [0,1]$. Such an alternative way, however, seems to be less appealing. A rule with a set of assumptions that is not even countable might be considered as not suitable for the present context.

\subsection{Adding mean-values}

We now modify the logic \LGI{} so as to have the possibility to refer to mean values. The obtained logic will be called \LGIM{}. 

A {\it generalised graded implication} is a triple consisting of a multiset $\alpha_1, \ldots, \alpha_n$ of basic expressions, a further basic expression $\beta$, and a real $c \in [0,1]$; we write
\begin{equation} \label{fml:generalised-graded-implication}
\alpha_1, \ldots, \alpha_n \implc{c} \beta.
\end{equation}
A formula of \LGIM{} is defined similarly as a formula of \LGI, but this time built up from generalised graded implications.

For real numbers $r_1, \ldots, r_n$, $n \geq 1$, let us write
\[ \average{r_1, \ldots, r_n} \;=\; \frac{r_1 + \ldots + r_n}n. \]
We define the generalised graded implication (\ref{fml:generalised-graded-implication}) to be satisfied by some evaluation $v$ if
\begin{equation} \label{fml:generalised-graded-implication-satisfied}
\average{v(\alpha_1), \ldots, v(\alpha_n)} \leq v(\beta) + \stneg c.
\end{equation}
The satisfaction of formulas and the semantic entailment relation are defined for \LGIM{} similarly as for \LGI, but such that (\ref{fml:generalised-graded-implication-satisfied}) is taken into account. We denote the entailment relation again by $\entails$.

\begin{definition}
The calculus \LGIM{} consists of all axioms and rules belonging to \LGI{} as well as the following axioms, for all basic expressions $\alpha, \alpha_1, \ldots, \alpha_n, \beta_1, \ldots, \beta_n, \gamma$ and any $c, c_1, \ldots, c_n, d \in [0,1]$:
\begin{align*}
& \Rule{trans$\averagealone_1$} (\alpha_1 \implc{c_1} \beta_1) \land \ldots 
\land (\alpha_n \implc{c_n} \beta_n) 
\land (\beta_1, \ldots, \beta_n \implc{d} \gamma) \impl
(\alpha_1, \ldots, \alpha_n \implc{\average{c_1, \ldots, c_n} \Ltnorm d} \gamma) \\
& \Rule{trans$\averagealone_2$} (\alpha_1, \ldots, \alpha_n \implc{c} \beta) \land (\beta \implc{d} \gamma) \impl
(\alpha_1, \ldots, \alpha_n \implc{c \Ltnorm d} \gamma) \\
& \Rule{trans$\averagealone_3$} (\alpha_1 \implc{c_1} \false) \land \ldots 
\land (\alpha_n \implc{c_n} \false) 
\land (\true \implc{d} \beta) \impl
(\alpha_1, \ldots, \alpha_n \implc{\average{c_1, \ldots, c_n} \Ltconorm d} \beta) \\
& \Rule{$\true\averagealone$} (\true, \ldots, \true \implc{c} \alpha) \impl (\true \implc{c} \alpha).
\end{align*}
We write ${\mathcal T} \provesLGIM \Phi$ if there is a proof of $\Phi$ from $\mathcal T$ in \LGIM.
\end{definition}

\begin{theorem} \label{thm:soundness-LGIM}
Let $\mathcal T$ be a theory and $\Phi$ a formula of \LGIM. If ${\mathcal T} \provesLGIM \Phi$, then ${\mathcal T} \entails \Phi$.
\end{theorem}

\begin{proof}
It is not difficult to check that the additional rules are sound. Hence the assertion follows from Theorem \ref{thm:soundness-LGI}.
\end{proof}

To adapt our completeness theorem, we need to extend Lemma \ref{lem:explicit-values-fulfil-graded-implication} to include generalised graded implications.

\begin{lemma} \label{lem:explicit-values-fulfil-graded-implication-with-mean}
Let $\mathcal T$ be a theory, let $\alpha_1 \ldots, \alpha_n, \beta$ be basic expressions, and let $c_1, \ldots, c_n, d \in [0,1]$. Assume that $\tau(\alpha_1,c_1)$, \ldots, $\tau(\alpha_n,c_n)$ and $\tau(\beta,d)$ are provable from $\mathcal T$. Then ${\mathcal T} \provesLGIM \alpha_1, \ldots, \alpha_n \implc{r} \beta$ if $r < 1 - \average{c_1, \ldots, c_n} + d$, and ${\mathcal T} \provesLGIM \neg(\alpha_1, \ldots, \alpha_n \implc{r} \beta)$ if $r > 1 - \average{c_1, \ldots, c_n} + d$.
\end{lemma}

\begin{proof}
We proceed in analogy to the proof of Lemma \ref{lem:explicit-values-fulfil-graded-implication}, using this time the rules (trans$\averagealone_1$)--(trans$\averagealone_3$) and ($\true\averagealone$).
\end{proof}

We are now in the position to show a Pavelka-style completeness theorem for \LGIM.

\begin{theorem} \label{thm:completeness-LGIM}
Let $\mathcal T$ be a theory consisting of generalised graded implications, and let $\zeta_1, \ldots, \zeta_n \implc{e} \eta$ be a generalised graded implication of \LGIM. If ${\mathcal T} \entails \zeta_1, \ldots, \zeta_n \implc{e} \eta$, then ${\mathcal T} \provesLGIM \zeta_1, \ldots, \zeta_n \implc{t} \eta$ for any $t < e$.
\end{theorem}

\begin{proof}
Given Lemma \ref{lem:explicit-values-fulfil-graded-implication-with-mean}, we can argue in analogy to the proof of Theorem \ref{thm:completeness-LGI}.
\end{proof}

\subsection{Application of \LGIM{} to questionnaires}

We now show how the intended reasoning in the context of questionnaire score calculation can be reproduced in the logic \LGIM.

We choose the variables of \LGIM{} such that each one refers to a clinical entity, such as a symptom or a disorder. For formulas of the form $\true \implc{c} \alpha$ or $\alpha \implc{\stneg c} \false$, where $\alpha$ is a variable and $c \in [0,1]$, the intended meaning is that $\alpha$ applies to a patient to the degree at least or at most $c$, respectively.

Let us assume that we are given a questionnaire containing $n \geq 1$ items. We consider each item as a symptom; let $\phi_1, \ldots, \phi_n$ be the corresponding variables. Let us furthermore assume that the questionnaire assesses one disorder; let $\delta$ be the variable referring to it.

We want to express in \LGIM{} that the truth degree of $\delta$ is the average of the degrees of $\phi_1, \ldots, \phi_n$. This is quite straightforward. The formula
\begin{equation} \label{fml:average-lower-bound}
\phi_1, \ldots, \phi_n \implc{1} \delta
\end{equation}
is satisfied by some evaluation $v$ if and only if $\average{v(\phi_1), \ldots, v(\phi_n)} \leq v(\delta)$. Furthermore,
\begin{equation} \label{fml:average-upper-bound}
\stneg\phi_1, \ldots, \stneg\phi_n \implc{1} \stneg\delta
\end{equation}
is satisfied by $v$ if and only if $\frac 1 n ((1-v(\phi_1)) + \ldots + (1-v(\phi_n))) \leq 1 - v(\delta)$ if and only if $1-\average{v(\phi_1), \ldots, v(\phi_n)} \leq 1-v(\delta)$ if and only if $\average{v(\phi_1), \ldots, v(\phi_n)} \geq v(\delta)$. Consequently, letting the theory $\mathcal T$ consist of (\ref{fml:average-lower-bound}), (\ref{fml:average-upper-bound}), as well as
\begin{align*}
\phi_1 \implc{\stneg c_1} \false, \;\; \ldots, \;\; \phi_n \implc{\stneg c_n} \false, \;\;
\true \implc{c_1} \phi_1, \;\; \ldots, \;\; \true \implc{c_n} \phi_n,
\end{align*}
we can derive from $\mathcal T$ the graded implications
\[ \delta \implc{\stneg d} \false, \quad \true \implc{d} \delta, \]
where $d$ is the mean value of $c_1, \ldots, c_n$.

In fact, from $\true \implc{c_1} \phi_1, \;\; \ldots, \;\; \true \implc{c_n} \phi_n$ and (\ref{fml:average-lower-bound}), we get $\true, \ldots, \true \implc{\average{c_1, \ldots, c_n}} \delta$ by (trans$\averagealone_1$) and $\true \implc{\average{c_1, \ldots, c_n}} \delta$ by ($\true\averagealone$), that is, $\true \implc{d} \delta$. Similarly, we get $\true \implc{\stneg c_1} \stneg\phi_1, \;\; \ldots, \;\; \true \implc{\stneg c_n} \stneg\phi_n$ from ($\stneg_1$) and proceed similarly as before to derive $\true \implc{\average{\stneg c_1, \ldots, \stneg c_n}} \stneg\delta$ and thus $\delta \implc{\average{\stneg c_1, \ldots, \stneg c_n}} \false$, that is, $\delta \implc{\stneg d} \false$.

\section{A logic of prototypes and counterexamples}
\label{sec:prototypes-counterexamples}

We have shown that the fuzzy logic \LGIM{} presented in the previous section is suitable to emulate the calculations that are performed for the evaluation of certain medical assessment questionnaires. Consequently, \LGIM{} could be regarded as an appropriate answer to our concern of defining a logical framework for this particular medical context. The approach has, however, certain weaknesses. First of all, we should admit the arbitrariness of several design choices; this is a problem that we encounter in mathematical fuzzy logic inevitably. Second, while extending this framework to other scoring methods than the mean-value is possible in principle, it implies  significant modifications. Finally -- and this is a subjective issue -- the framework of mathematical fuzzy logic may not be found very intuitive and thus not very appealing for practitioners designing medical assessment questionnaires. In particular, we are not likely to explain the meaning of the statements and derivations made in \LGIM{} to somebody without an appropriate background in many-valued logic.

In reply to these arguments, we investigate in this section the question if the same result -- the emulation of calculation of assessment scores -- can be obtained in a framework that addresses the characteristic features of the application in a more appropriate way. We have in particular the aspect of vagueness in mind. The alternative approach that we introduce in the sequel is based on the idea of modelling the prototypes of a vague property by a subset of a metric space \cite{dubois1997three}. Certain ideas originating from approximate reasoning \cite{godo2008logical,esteva2012strong} play a role as well.

\subsection{Modelling vague properties in metric spaces}

A vague property, like ``tall'' for a human being, refers to a size, but does not correspond to a partition of the set of all sizes. In fact, the notion ``tall'' refers to a lower level of granularity than the elements of the real interval $[0,250]$ when used to indicate sizes in centimeters. In particular, the property of being ``tall'' cannot be identified with a subset of $[0,250]$ because there is no smallest size to be considered as tall.

It might be less difficult, however, to choose a subset of $[0,250]$ that is not supposed to consist of all sizes to be considered ``tall'' but only those unquestionably to be considered ``tall''; we speak about the {\it prototypical} cases then. In addition, we can endow the set $[0,250]$ of all sizes with a similarity relation such that increasing distance from the set of prototypes reflects the decreasing tendency to call somebody with that size tall.

This idea of modelling vague notions can be regarded as fundamental in fuzzy set theory. In the seminal paper \cite{dubois1997three}, three different views on fuzzy sets are specified and one of these views coincides with our present approach. A series of formalisms based on the idea were developed; we may mention, e.g., \cite{lawry2009prototype}. It can certainly be argued that the choice of prototypes involves arbitrariness as well. When combining two levels of granularities, arbitrariness can in fact never be avoided. However, we avoid that a choice about the truth degree must be made for every element of the base set.

The similarity-based approach to fuzzy sets has, however, a disadvantage. The degree assigned to a particular element of the base set depends on its distance from the set of prototypes; consequently, all fuzzy sets are ``equally steep''. This is counterintuitive. Indeed, it makes a difference if the set of prototypes is large and models some very general property, or the set is small and models a more specific property. The neighborhood of the set of prototypes that contains the elements mapped to non-zero degrees should in the latter case be narrower than in the former.

There is a straightforward way to overcome this difficulty. A vague property is not only characterised by its prototypes, but also by its counterexamples. By a counterexample, we mean a case in which the property under consideration clearly does not hold. In this case, we model properties by pairs of sets, the first of which contains the prototypes, the second one the counterexamples. We are led to a model based on disjoint subsets; see, e.g., \cite{ciucci2014threevalued} for a detailed discussion.

This idea is still well in line with fuzzy set theory. It belongs, e.g., to the essential constituents in V.\ Nov\' ak's theory of trichotomous evaluative linguistic expressions \cite{novak2008trichotomous,novak1999principlesfuzzylogic}. Recall that a fuzzy set modelling a vague property maps the prototypical elements to $1$ and the counterexamples to $0$. Given the prototypes and counterexamples, the borderline cases can moreover be handled by means of a metric structure. In the linguistic context, an S-shaped fuzzy set has been established as most appropriate. Often, however, the remaining grades are determined on the basis of a simple linear interpolation: then each element is mapped to the distance from the counterexamples divided by the sum of the distances from the prototypes and counterexamples.

The idea of identifying properties with the sets of their prototypes and counterexamples might be found appealing; there is, however, no straightforward way to design on its basis a method of reasoning about vague properties. One point seems to be clear: fuzzy logic understood as t-norm-based many-valued logic is not suitable. Below we demonstrate how reasoning in the questionnaire context can be represented along the indicated lines.

\subsection{A simple prototype-counterexample logic}
\label{sec:logic-for-questionnaires}

For the sake of the considered application, it is natural to identify the universe of discourse with the totality of possible outcomes of a questionnaire-based interview and thus with the set of all truth value assignments of a certain finite set of variables. We specify in the sequel a formalism on this basis, which we shall denote by \Q.

\Q{} is specified syntactically as follows. We fix $n \geq 1$ and we let the symbols $\phi_1, \ldots, \phi_n$ be our \emph{basic variables}. Furthermore, we choose a countably infinite set $\alpha_1, \alpha_2, \ldots$ of \emph{dependent variables}. All variables are assumed to model vague properties and are hence subject to an assignment with graded truth degrees. This will be done in the following explicit way. Again, we denote by $[0,1]$ the real unit interval. A \emph{graded variable} is an expression of the form $(\alpha, c)$, where $\alpha$ is a variable and $c \in [0,1]$.

A \emph{formula} of \Q{} is built up from graded variables by means of the binary connectives $\land, \lor$ and the unary connective $\lnot$. Their meaning is the classical ``and'', ``or'', and ``not'', respectively. Note that we do not compose the variables themselves by logical connectives but only variables endowed with a degree.

The semantics of \Q{} is based on the following considerations. As common in degree-based propositional logic, the variables together with an assignment of truth values are intended to describe one out of a set of possible situations; technically, we speak of possible worlds. Here, we will assume that each possible world is specifiable by the truth degrees of the basic variables alone. That is, we assume that there is a one-to-one correspondence between the assignments of the basic variables with truth values and the set of worlds. The truth degree of each dependent variable is in turn assumed to depend on the truth degree of the basic variables.

Accordingly, we proceed as follows. We define $W = [0,1]^n$ to be the \emph{set of worlds}. We measure the distance between two worlds by the sum of the differences of the $n$ components. That is, we endow $W$ with the metric $d \colon W \times W \to \Reals^+$ defined by
\begin{equation} \label{fml:metric}
d( (c_1, \ldots, c_n), (d_1, \ldots, d_n) ) \;=\; |c_1-d_1| + \ldots + |c_n-d_n|,
\end{equation}
where $c_1, \ldots, c_n, d_1, \ldots, d_n \in [0,1]$. We additionally define the distance of a $w \in W$ from some $A \subseteq W$ by $d(w,A) = \inf\; \{ d(w,a) \colon a \in A \}$.

$W$ is our fixed domain of interpretation. An \emph{evaluation} maps each variable $\alpha$ to a pair $(\alpha^+, \alpha^-)$ of disjoint non-empty closed subsets of $W$. For each basic variable $\phi_i$, we require that
\begin{align*}
& \phi_i^+ \;=\; \{ (a_1, \ldots, a_n) \colon a_i = 1 \}, \\
& \phi_i^- \;=\; \{ (a_1, \ldots, a_n) \colon a_i = 0 \}.
\end{align*}
The intended meaning is that $\alpha^+$ contains the prototypes and $\alpha^-$ contains the counterexamples of $\alpha$.

We can then associate with $\alpha$ a fuzzy set according to the approach outlined above. To this end, we consider the distance of a world $w$ from the set of prototypes as well as from the set of counterexamples. Namely, at $w \in W$, we say that $\alpha$ holds to the degree
\[ 1 - \frac{d(w,\alpha^+)}{d(w,\alpha^+) + d(w,\alpha^-)} \;=\; \frac{d(w,\alpha^-)}{d(w,\alpha^+) + d(w,\alpha^-)}.\]
Note that $\alpha$ holds to the degree $1$ if and only if $w \in \alpha^+$, and $\alpha$ holds to the degree $0$ if and only if $w \in \alpha^-$. Furthermore, the basic variable $\phi_i$ holds at a world $(a_1, \ldots, a_n)$ to the degree $a_i$. 

Given an evaluation, we assign to each formula $\Phi$ a subset $[\Phi]$ of $W$ as follows. For each graded variable $(\alpha, c)$, we define
\begin{equation} \label{fml:evaluation-graded-variable}
[(\alpha, c)] \;=\; \{ w \in W \colon \tfrac{d(w,\alpha^-)}{d(w,\alpha^+) + d(w,\alpha^-)} = c \},
\end{equation}
and for a compound formula $\Phi$, we define $[\Phi]$ such that the connectives $\land, \lor, \lnot$ are interpreted by the set-theoretic operations $\cap, \cup, \complement$, respectively. The evaluation is said to {\it satisfy} $\Phi$ if $[\Phi] = W$.

In other words, $[(\alpha, c)]$ consists of those worlds at which $\alpha$ holds to the degree $c$. In particular, for each basic variable $\phi_i$ and $t \in [0,1]$ we have
\[ [(\phi_i,t)] \;=\; \{ (a_1, \ldots, a_n) \in W \colon a_i = t \}. \]
Moreover, a compound formula is satisfied if the corresponding set-theoretical relationship holds. For instance, $[(\phi_1,1) \land (\phi_2,0) \impl (\alpha,1)]$ is satisfied iff  $[(\phi_1,1)] \cap [(\phi_2,0)] \subseteq [(\alpha,1)]$ iff, at every world at which $\phi_1$ holds to the degree $1$ and $\phi_2$ holds to the degree $0$, $\alpha$ holds to the degree $1$.

Finally, a \emph{theory} of \Q{} is a set of formulas. A theory $\mathcal T$ is said to be {\it correct} if there is an evaluation satisfying all elements of $\mathcal T$. We say that a correct theory \emph{entails} a formula $\Phi$ if every evaluation satisfying all elements of $\mathcal T$ also satisfies $\Phi$.

The role of theories of \Q{} may be characterised as follows. Their scope is to specify the dependent variables relative to the independent ones, that is, to determine the truth degree of each dependent variable given the truth degrees of the independent ones. The framework is given by the set $W$ and the metric (\ref{fml:metric}) defined on it. By means of a theory of \Q, we are supposed to specify the sets of prototypes and counterexamples, from which the remaining truth degrees are determined by the interpolative prescription (\ref{fml:evaluation-graded-variable}).

We note that we are in this way led to a viewpoint that differs from the common procedure, e.g., in mathematical fuzzy logic, where truth degrees are seen relative to each other at each single world. In contrast, our approach determines truth values at a given world by reference to truth values at other worlds. Accordingly, we do not allow to exclude worlds, or to consider specific worlds without the remaining ones; we always consider $W$ as a whole. Hence we are interested in those evaluations that assign each element of a theory the whole set $W$. To ensure that such evaluations exist, theories are required to be correct.

We may conclude that our approach assigns to theories a role that is narrower than in other logics. Often, the general facts, like in our context the interpretation of dependent variables, and the special facts, like the propositions holding at a specific world, are not formally distinguished and can both be included in a theory. In \Q, theories are reserved for general facts, referring to the whole fixed set $W$. To reason about specific situations, for instance about the case that a variable holds to a certain degree, we use compound formulas. The example provided in the next subsection will demonstrate that this procedure is actually practicable.

\subsection{Application of \Q{} to questionnaires}
\label{sec:application-to-questionnaires}

We now demonstrate the reasoning in the context of questionnaires in the proposed framework. 

As before, the variables of \Q{} are intended to refer to clinical entities. For a variable $\alpha$ and $c \in [0,1]$, the expression $(\alpha, c)$ means that $\alpha$ applies to a patient to the degree $c$. If $\alpha$ is not vague, $c$ always equals $0$ or $1$.

Assume that we are given a questionnaire containing $n \geq 1$ items and each item can be answered with one out of $k+1$ degrees, $k \geq 1$. We choose the basic variables $\phi_1, \ldots, \phi_n$ in correspondence with the items of the questionnaire. The dependent variables describe properties depending on what the basic variables refer to; we assume that this is case for the syndromes, diseases, or disorders under consideration. Here, we assume again that a single disorder $\delta$ is tested.

We furthermore assume that answers can appear in all combinations. Consequently, it makes sense to define the set of worlds as we do above, by $W = [0,1]^n$. The set of all possible answers certainly corresponds to a finite subset of $W$, namely, $\{ 0, \tfrac 1 k, \ldots, 1 \}^n$.

In this framework, let us specify the disorder $\delta$ under consideration. The items of a questionnaire are chosen such that $\delta$ is fully confirmed if all of them are answered clearly affirmatively. Accordingly, we let the set of prototypes of $\delta$ be the singleton
\[ \delta^+ \;=\; \{ (1, \ldots, 1) \}. \]

The reason for this particular choice of the set of prototypes is the intended correspondence to the context of medical questionnaires, in which the total 
value of 1 is only possible in case the  answers to all items are 1.

Similarly, the disorder is fully excluded only if all questions are answered negatively. In particular, this conclusion is not supposed to be drawn if only some of the answers are negative. Accordingly, we let the set of counterexamples of $\delta$ again consist of only one element:
\[ \delta^- \;=\; \{ (0, \ldots, 0) \}. \]
A world $w = (c_1, \ldots, c_n)$ corresponds to a particular patient answer. We now see that $\delta$ is at $w$ assigned the expected degree, namely,
\begin{equation} \label{fml:patient-answer}
\begin{split}
& \frac{d(w,\delta^-)}{d(w,\delta^+) + d(w,\delta^-)} \\
& =\; \frac{d((c_1, \ldots, c_n),(0,0,0,0))}{d((c_1, \ldots, c_n),(1,1,1,1)) + d((c_1, \ldots, c_n),(0,0,0,0))} \\
& =\; \frac{c_1 + \ldots + c_n}{(1-c_1) + \ldots + (1-c_n) + c_1 + \ldots + c_n} \\
& =\; \tfrac 1 n (c_1 + \ldots + c_n).
\end{split}
\end{equation}
that is, the arithmetic mean of the grades assigned to the items.

To see that \Q{} in fact emulates the calculation of scores from given degrees, let $n = 4$ and denote by $\phi_1, \phi_2, \phi_3, \phi_4$ the four items characterising the syndrome $\delta$ ``depression''. We need a theory specifying $\delta$ for all possible answers to these four items. The syndrome $\delta$ is specified by its sets of prototypes and counterexamples; accordingly, let $\mathcal T$ contain the following two formulas:
\begin{align*}
& (\delta,1) \;\leftrightarrow\; (\phi_1,1) \land (\phi_2,1) \land (\phi_3,1) \land (\phi_4,1), \\
& (\delta,0) \;\leftrightarrow\; (\phi_1,0) \land (\phi_2,0) \land (\phi_3,0) \land (\phi_4,0),
\end{align*}
where $\leftrightarrow$ has the usual meaning. Assume that these two formulas are satisfied by an evaluation. This means
\begin{align*}
& [(\delta,1)] \;=\; [(\phi_1,1)] \cap [(\phi_2,1)] \cap [(\phi_3,1)] \cap [(\phi_4,1)], \\
& [(\delta,0)] \;=\; [(\phi_1,0)] \cap [(\phi_2,0)] \cap [(\phi_3,0)] \cap [(\phi_4,0)];
\end{align*}
that is, $\delta^+ = \{ (1,1,1,1) \}$ and $\delta^- = \{ (0,0,0,0) \}$. Furthermore, $[(\delta,c)]$ is uniquely determined for each $c \in [0,1]$ by (\ref{fml:patient-answer}).

Let now $c_1, c_2, c_3, c_4$ be the answers provided by a patient. We are interested in deriving the consequences of this special fact within \Q. To this end, we explore which implications of the form
\[ (\phi_1, c_1) \wedge (\phi_2, c_2) \wedge (\phi_3, c_3) \wedge (\phi_4, c_4) \impl \Phi \]
$\mathcal T$ entails. We have
\begin{align*}
& [(\phi_1, c_1) \wedge (\phi_2, c_2) \wedge (\phi_3, c_3) \wedge (\phi_4, c_4)] \\
& \;=\; [(\phi_1, c_1)] \cap [(\phi_2, c_2)] \cap [(\phi_3, c_3)] \cap [(\phi_4, c_4)] \;=\;
\{ (c_1, c_2, c_3, c_4) \}
\end{align*}
and by (\ref{fml:patient-answer})
\[ (c_1, c_2, c_3, c_4) \in [(\delta, \tfrac{c_1 + c_2 + c_3 + c_4} 4)]. \]
Hence
\[ [(\phi_1, c_1) \wedge (\phi_2, c_2) \wedge (\phi_3, c_3) \wedge (\phi_4, c_4)]
 \;\subseteq\; [(\delta, \tfrac{c_1 + c_2 + c_3 + c_4} 4 )], \]
and we conclude that $\mathcal T$ entails
\[ (\phi_1, c_1) \wedge (\phi_2, c_2) \wedge (\phi_3, c_3) \wedge (\phi_4, c_4) \;\impl\; (\delta, \tfrac{c_1 + c_2 + c_3 + c_4} 4 ). \]
Hence, as desired, the particular questionnaire outcome $c_1, c_2, c_3, c_4$ implies that $\delta$ holds to the degree $\tfrac{c_1 + c_2 + c_3 + c_4} 4$.

\begin{remark}\label{remarkdimension}

The model presented above may be illustrated as follows. A patient's answers $\chi$ to a questionnaire may be visualised by a point in a multi-dimensional space, where each dimension corresponds to one item of the questionnaire. Furthermore, the counterexamples and prototypes of the disorder in question are represented by points in this space as well. Finally, the obtained score is the relative distance of $\chi$ from the counterexamples and prototypes; cf.\ {\rm (\ref{fml:evaluation-graded-variable})}. Assuming only three items, the situation is expressed in Figure {\rm \ref{fig}}.

\begin{figure}[h]
\begin{center}

\includegraphics[scale=0.3]{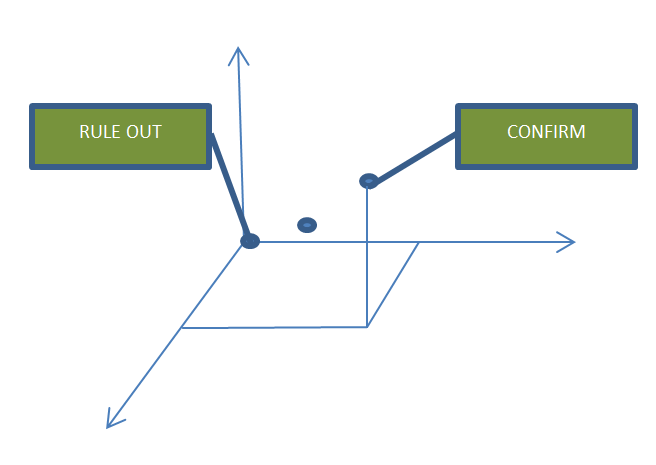} 
\end{center}
\caption{Illustration of the distance of $\chi$, representing a patient's answers, from the counterexamples and prototypes of a disorder.} \label{fig}
\end{figure}

\end{remark}

\section{Conclusion}
\label{sec:conclusion}

Representing apparently straightforward medical reasoning within a logical framework is generally a tricky task. Principle limits have been observed in case of fuzzy-logic based medical decision support systems; see, for instance, \cite{ciabattoni2013formalapproaches}. In this paper we address a specific further problem in the medical domain: providing a formal framework for the score calculation in assessment questionnaires. We have included two very different approaches following the same aim. Roughly speaking, we may summarise our efforts as follows: putting up an elegant formal system leads to difficulties in interpretation; restricting to well-justified principles leads to narrow constraints on the formal side.

At the end, in the medical context, the more valuable approach is the one that is more useful in practice. Thus, an evaluation of the proposed approaches is the most straightforward direction for further research.

On the more theoretical side, the first approach is basically a fuzzy logic that is, in line with our application, able to deal with mean values. Apart from that, our logic is characterised by the fact that the implication is present but not as a connective in the usual sense. The implication is, so-to-say, crispified by the attachment of an explicit truth degree.

The idea underlying the second part of the present paper was to represent vague properties by pairs of sets in a metric space; in this way, prototypes and counterexamples are modelled separately and the truth degrees are determined by the underlying metric. This approach provides an intuitively appealing model of vague notions in general and of notions occurring in the context of medical questionnaires in particular. It might be interesting to note that the independency of the questionnaire items that we have assumed for our model -- cf.\ Remark \ref{remarkdimension} -- is also a central basic assumption made in the majority of statistical models developed for analyzing questionnaires, such as classical test theory and item responce theory (\cite{reckase1997past,mcdonald2000basis,devellis2006classical}). Exploring further connections to these models seems a fruitful direction for further research.

As regards the second approach, the most important issue for further research is an axiomatisation. To define a nice proof system for a logic of this kind remains a serious, although possibly quite rewarding, challenge. A possible way to go might be to subject the pairs of subsets modelling vague properties to additional constraints.

\subsubsection*{Acknowledgements.} 
The first author was supported by the Austrian Science Fund (FWF): project I 1923-N25 (New perspectives on residuated posets). The second author was supported by The Israel Science Foundation under grant agreement no. 817/15. 

We would moreover like to express our gratitude to the reviewers, whose constructive criticism led to an improvement of the results of this paper.


\end{document}